\documentclass[
a4paper,
twoside,
10pt
]{amsart}
\usepackage[pdftex,
colorlinks, 
bookmarksnumbered, 
bookmarksopen, 
linkcolor=blue, 
citecolor=blue, 
]{hyperref}
\usepackage{eucal}
\usepackage{amssymb}
\usepackage {amsmath}
\usepackage{mathptmx}
\providecommand {\norm}[1] {\lVert#1\rVert}

\providecommand {\abs}[1] {\lvert#1\rvert}

\providecommand {\inprod}[1]{\langle #1 \rangle}
\providecommand {\set}[1]{\lbrace #1 \rbrace}
\providecommand {\bigset}[1]{\Bigl\lbrace #1 \Bigr\rbrace}
\providecommand {\floor}[1]{\lfloor #1 \rfloor}
\providecommand {\besov}[3]{\Lambda^ {#1} _ {#2} ({#3})}
\providecommand {\cc}[2][p]{\mC^{#1}_{#2}}

\providecommand {\dadas}[3]{D^ {#1} _ {#2} ({#3})}
\providecommand {\dada}[2][M]{D^{1}_{#1} ({#2})} 
\providecommand {\app}[3]{\mathcal{E}^{#1}_{#2} ({#3})}
\providecommand {\inject}{\hookrightarrow}
\providecommand {\inv}[1]{{#1}^{-1}}

\providecommand {\domd}[1] [\mA] {\mD(#1)}

\newcommand {\one} {\mathbf{1}}
\newcommand {\bn} {\ensuremath{\mathbb{N}}}
\newcommand {\bno} {\ensuremath{\mathbb{N}_0}}
\newcommand {\br} {\ensuremath{\mathbb{R}}}
\newcommand {\bz} {\ensuremath{\mathbb{Z}}}
\newcommand {\bt} {\ensuremath{\mathbb{T}}}

\newcommand {\bnd} {\ensuremath{\mathbb{N}^d}}
\newcommand {\brd} {\ensuremath{\mathbb{R}^d}}
\newcommand {\btd} {\ensuremath{\mathbb{T}^d}}
\newcommand {\bcd} {\ensuremath{\mathbb{C}^d}}

\newcommand {\bzd}[1] [d] {\ensuremath{\bz^{#1}}}
\newcommand {\mA} {\ensuremath{\mathcal{A}}}
\newcommand {\mD} {\ensuremath{\mathcal{D}}}
\newcommand {\mE} {\ensuremath{\mathcal{E}}}

\newcommand {\mJ} {\ensuremath{\mathcal{J}}}
{

\newcommand {\mB} {\ensuremath{\mathcal{B}}}

\newcommand {\mC} {\ensuremath{\mathcal{C}}} 

\newcommand {\mT} {\ensuremath{\mathcal{T}}}

\newcommand {\lp} [1] [p] {\ensuremath{\ell^{#1}(\bzd)}}  
\newcommand {\lpw}[2] [p] {\ensuremath{\ell^{#1}_{#2}(\bzd)}}  

\newcommand {\wlp}[2] [p] {\ensuremath{\ell_{#2}^#1}} 

\newcommand {\bopzd}  {{\ensuremath{\mathcal B ( \ell^2(\bzd))}}}
\newcommand {\bopz}  {\ensuremath{\mathcal B ( \ell^2(\bz))}}
\newcommand {\bop}  {{\ensuremath{\mathcal B ( \ell^2)}}}

\DeclareMathOperator{\diag}{Diag}

\DeclareMathOperator{\id}{\mathrm{id}}

\newcommand \BL {bandlimited}
\newcommand \tfae {the following are equivalent}
\newcommand \smult {submultiplicative}

\newcommand \cont {continuous}
\newcommand \BS {Banach space}
\newcommand \BA {Banach algebra}
\newcommand \MA {matrix algebra}

\newcommand \HMA {homogeneous matrix algebra}
\newcommand \hmg {homogeneous}

\newcommand \odd {off-diagonal decay}

\newcommand \as {approximation space}
\newcommand \IC {inverse-closed}

\newcommand {\cexp} [1] [{x t}] {\ensuremath{e^{2 \pi i #1}}}

\newcommand { \muleb} [1] {\frac{d {#1}}{#1}}

 \newcommand {\bigo}{{\mathcal{O}}}

\newtheorem{prop}{Proposition} 
\newtheorem{cor}[prop]{Corollary}
\newtheorem{thm}[prop]{Theorem}

\newtheorem{lem}[prop]{Lemma}

\theoremstyle{definition}
\newtheorem{defn}[prop]{Definition} 

\theoremstyle{remark}
\newtheorem*{rem}{Remark}

\newtheorem{ex}[prop]{Example}

\newcommand \Wlog {Without loss of generality}
 \newcommand \wwlog {without loss of generality}
\begin{document}
\title{Inverse Closed Ultradifferential Subalgebras}
\date{\today} 
\author{ Andreas Klotz }
\address{Faculty of Mathematics\\University of Vienna \\
   Nordbergstrasse 15\\A-1090 Vienna, AUSTRIA} \email{andreas.klotz@univie.ac.at} 
\thanks{ A.~K.~ was supported by  National Research Network S106 SISE of the Austrian Science Foundation (FWF) and the FWF project P22746N13}

\begin{abstract}
In previous work we have shown that 
classical approximation theory provides methods for the systematic construction of \IC\ smooth subalgebras.
Now we extend this work to treat \IC\ subalgebras of ultradifferentiable elements. In particular, Carleman classes and Dales-Davie algebras are treated. As an application the result of Demko, Smith and Moss and Jaffard on the inverse of a matrix with exponential decay is obtained within the framework of a general theory of smoothness.
\end{abstract}
\maketitle
\section{Introduction}
\label{sec:introduction}

We describe new methods to generate a smooth \IC\ subalgebra of a given \BA\ \mA\ and to characterize this subalgebra by approximation properties and by weights.
Recall that a subalgebra \mB\ of \mA\ is \IC\ in \mA, if
\begin{center}
  every $b \in \mB$ that is invertible in \mA\ is actually invertible in \mB .
\end{center}
A prototype of an \IC\ subalgebra is the Wiener algebra of absolutely convergent Fourier series, which is \IC\ in the algebra of \cont\ functions on the torus. Another example is the algebra $C^1(\bt)$ of  \cont ly differentiable functions on the torus; the proof that $C^1(\bt)$ is \IC\ in $C(\bt)$ is essentially the  quotient rule of classical analysis.

Many methods  for the construction of \IC\ subalgebras are based on generalizations of this simple smoothness principle. In the context of  \BA s, derivatives are replaced by derivations. 
The Leibniz rule for derivations implies that their domain is a \BA, and by the symmetry of \mA\  the domain is \IC\ in \mA, see \cite{grkl10}.

A more refined concept of smoothness can be developed, if \mA\ is invariant under the bounded action of a $d$-dimensional automorphism group. In this case algebras of Bessel-Besov type can be defined, and the properties of the group action imply that the spaces defined form \IC\ subalgebras of \mA, see~\cite{klotz10a}.

A different approach to smoothness is by approximation using approximation schemes adapted to the algebra multiplication. This line of research, initiated by Almira and Luther~\cite{Almira01,Almira06}, yields \BA s of approximation spaces that are \IC\ in \mA, if \mA\ is symmetric~\cite{grkl10}. 

Moreover, if \mA\ is invariant under the action of the translation group and the approximation scheme consists of the \emph{bandlimited elements} of \mA, 
 we obtain Jackson-Bernstein theorems that identify approximation spaces of polynomial order with Besov spaces.

All of the above has been carried out in two previous publications~\cite{grkl10,klotz10a} for smoothness spaces of finite order. Now we use the same principles to construct \IC\ subalgebras of ultradifferentiable elements.

Classes of Carleman type are defined by growth conditions on the norms of higher derivations in the same way as  for functions, and we obtain a characterization of  \IC\ Carleman classes  by adapting a proof of Siddiqi~\cite{Siddiqi90}. 
If the growth of the derivations satisfies the condition (M2') of Komatsu,  then an alternative description of the Carleman classes as union of weighted spaces or approximation spaces is available.

Whereas  Carleman algebras are inductive limits of \BS s we can also define  \emph{\BA s}  of ultradifferentiable elements derived from a given \BA. The construction generalizes an approach used by Dales and Davie~\cite{dales73} for functions defined on perfect subsets of the complex plane, so we call the resulting \BA s \emph{Dales-Davie algebras}. An result of Honary and Abtahi~\cite{honary07} on \IC\ Dales-Davie algebras of functions can be adapted to the noncommutative situation (Theorem~\ref{thm-dada-is-IC}).

The general theory has applications to \BA s of matrices with  \odd. The formal commutator
$
\delta(A)= [X,A] 
$, 
$X=2 \pi i\diag((k)_{k \in \bz})$,  is a derivation on \bop, and its domain defines an algebra of matrices with \odd\ that is \IC\ in \bop~\cite[3.4]{grkl10}.
The translation group acts boundedly on \bop\ by conjugation with the modulation operator  
$
   M_t=\diag (\cexp [k \cdot t])_{k \in \bzd}
   $ ,
\begin{equation}
  \label{eq:1}
  \chi_t(A)= M_t A M_{-t} = \sum_{k \in \bzd} \hat A (k) \cexp [k \cdot t]  \, \quad \text{for  } t \in \brd \,,
\end{equation}
where  $\hat A(k)$ is the $k$th side diagonal of $A$,
\begin{equation}
    \hat A(k)(l,m)=\begin{cases}
      A(l,m),& \quad l-m=k,\\
      0,     & \quad \text{otherwise}.
    \end{cases}
 \end{equation}
In \cite{grkl10,klotz10a}  the  theory of smooth and \IC\ subalgebras has been applied to describe \BA s of matrices with \odd. 

The  approximation theoretic characterization of Carleman  Carleman classes of Gevrey type on \bop\ yields  a new proof of a result of Demko, Smith and Moss~\cite{Demko84}.
\begin{thm}
  If $A \in \bop$ with $\abs{{A(k,l)}} \leq C e^{-\gamma \abs {k-l}}$ for constants $C,\gamma >0$ and all $k,l, \in \bzd$, and if $\inv A \in \bop$, then there exist $C',\gamma'>0$ such that
  \[
\abs{\inv{A}}(k,l) \leq C' e^{-\gamma' \abs {k-l}} \quad \text{for all } k,l \in \bzd.
\]
\end{thm}

In some instances, Dales-Davie algebras of matrices can be identified with known \BA s of matrices, e.g.  if $\mC^1_{v_0}$  consists of matrices with norm
 \[
 \norm{A}_{\mC^1_{v_0}} =  \sum_{k \in \bzd} \sum_{l \in \bzd}\abs{A(l,l-k)} \,,
\]
 then  $\dadas 1 M {\mC^1_{v_0}}$ is a weighted form of this algebra for a submultiplicative weight $v_M$ associated to $M$, see Section~\ref{sec:dales-davie-classes}.
\\

The organization of the paper is as follows. First we  recall some facts from the theory of \BA s and review results of \cite{grkl10,klotz10a} on \IC\ subalgebras of a given \BA\ defined by derivations, automorphism groups, and approximation spaces.
In Section~\ref{sec:cinfty-ultr-class}, after treating $C^\infty$ classes,  ultradifferentiable classes of Carleman type are introduced, and necessary and sufficient conditions on their \IC ness are given.  Carleman classes satisfying axiom (M2') of Komatsu are characterized by approximation and weight conditions. As an application we generalize the result of Demko~\cite{Demko84} on the inverses of matrices with exponential \odd.
The results on the \IC ness of Dales-Davie algebras are treated in Section~\ref{sec:dales-davie-classes}.
In Section~\ref{sec:matrix-algebras-with} some applications to \MA s with \odd\ are given.
  In the appendix a combinatorial Lemma on the iterated quotient rule is proved.

\emph{Acknowledgment}: The author wants to thank Karlheinz Gr\"ochenig for many helpful discussions.

\section{Preliminaries}
\label{sec:preliminaries}
\subsection{Notation}
 \label{sec:notation}
The cardinality of a finite set $A$ is  $\abs A$.
The \emph{d-dimensional torus} is $\btd =\brd / \bzd $. 
The symbol $\floor x$ denotes the greatest integer smaller or equal to the real number $x$.
 Positive constants will be denoted by $C$, $C'$,$C_1$,$c$, etc., where the same symbol might denote different constants in each equation.

We use the standard multi-index notation. Multi-indices are denoted by Greek letters and are  a $d$-tuples of nonnegative integers. The degree of 
$x^\alpha=x_1^{\alpha_1}\cdots x_d^{\alpha_d}$ 
 is $\abs{\alpha} = \sum _{j=1}^d \alpha _j$, and $ D^\alpha f(x)= {\partial_1^{\alpha_1}}\cdots
{\partial_d^{\alpha_d}} f (x) $ is the partial derivative. The inequality  $\beta \leq \alpha $ means that $\beta _j \leq \alpha _j$ for all indices $j$.
 The $p$-norm on $\bcd$ is denoted by $\abs {x}_p=\bigl(\sum_{k=1}^d\abs {x(k)}^p \bigr)^{1/p}$

A submultiplicative weight   on $\bzd$ is a positive function $v:\bzd \to \br$ such that $v(0)=1$
and 
$v(x+y) \leq v(x)v(y)$ for $x,y \in \bzd$.  The standard  polynomial weights are
$v_r(x) = (1+\abs x)^r$
for $r \geq 0$. The weighted spaces $\lpw [p] w$ are defined by  the norm 
$\norm{x}_{\lpw [p]  w}=\norm{x w}_{\lp}$. If $w=v_r$ we  will simply write $\norm{x}_{\lpw [p]  r}$.
  A weight $w$ on \bzd\ satisfies the \emph{Gelfand, Raikov, Shilov (GRS)-condition} if
  $
  \lim_{n \to \infty} w(n \, x)^{1/n}=1$ for all $x \in \bzd$.

The \cont\ embedding of the normed space $X$  into  the normed space $Y$ is denoted as $X \inject Y$.
The operator norm of a bounded linear mapping $A \colon X \to Y$  is  $\norm{A}_{X \to Y}$. 
In the special case of operators $A \colon \lp [2] \to \lp [2]$ we write $\norm{A}_{\bopzd}=\norm{A}_{\lp [2] \to \lp
  [2]}$ or simply $\norm{A}_{\bop}$.

We will consider Banach spaces with equivalent norms as equal.

\subsection{Inverse closed  Banach algebras }
\label{sec:inverse-clos-matr}
All \BA s are assumed to be \emph{unital}. To verify that a \BS\ \mA\ with norm $\norm{\phantom{i}}_\mA$ is a \BA\ it is sufficient to
prove that $\norm{ab}_\mA \leq C \norm{a}_\mA \norm{b}_\mA$ for some constant $C$. 
A \BA\ \mA\ is a (Banach) $*$-algebra if it has an isometric  involution $*$, i.e., $\norm{a^*}_\mA=\norm{a}_\mA$ for all $a \in \mA$.
The Banach $*$-algebra $\mA$ is \emph{symmetric}, if 
$
\sigma_{\mA } (a^*a)\subseteq [0,\infty)
$ for all $a \in\mA$, where $\sigma_\mA(a)$ denotes the spectrum of $a \in \mA$. 
The spectral radius of $a \in \mA$ is $\rho_\mA(a)=\sup\set{\abs \lambda : \lambda \in \sigma_\mA(a)}$.

\begin{defn}[Inverse-closedness]
  If $\mA \subseteq \mB$ are  Banach algebras with common multiplication and identity, we call \mA\ \emph{inverse-closed}
  in \mB, if
  \begin{equation} \label{eq_3} a \in \mA \text{ and } a^{-1} \in \mB \quad \text{implies}\quad a^{-1} \in \mA.
  \end{equation}
\end{defn}
The relation of \IC ness is transitive: If \mA\ is \IC\ in \mB\, and \mB\ is \IC\ in \mC, then \mA\ is \IC\ in \mC.

\subsection{Derivations}
\label{sec:derivations}

  A \emph{derivation} $\delta$ on a \BA\ \mA\ with \emph{domain}
  $\mD=\mD(\delta)=\mD(\delta,\mA)$  a subspace of \mA\ is a closed linear
  mapping $\delta \colon \mD \to \mA$ that satisfies the Leibniz rule
  \begin{equation}
    \label{eq:derivation}
    \delta(ab) =a \delta(b) + \delta(a) b \qquad \text{for all} \,\,
    a,b \in \mD. 
  \end{equation}
  If \mA\ is a $*$-algebra, we assume that the derivation and
  the domain are symmetric, i.e., $\mD=\mD^*$ and $\delta(a^*)
  =\delta(a)^*$ for all $a \in \mD$. The domain is normed with the
  graph norm $\norm{a}_{\mD}=\norm{a}_\mA+\norm{\delta(a)}_\mA$.

Assume that  \mA\ is a symmetric \BA\ with
  a symmetric  derivation  $\delta$. 
  If $\one \in \domd$, then the (symmetric) \BA\ $\domd$ is \IC\ in \mA. Moreover, $\delta$ satisfies the quotient rule
  $
  \delta(a^{-1})=- a^{-1}\delta(a)a^{-1} 
  $, see ~\cite{grkl10}.

In more generality, let $\{\delta_1,\cdots, \delta_d\}$ be a set of commuting derivations
on  \mA.
The domain of $\delta _{r_1} \delta _{r_2} \dots
\delta _{r_n}$, $1\leq r_j \leq d$ is defined by induction as 
$
\mD (\delta _{r _1} \delta _{r _2} 
\dots \delta _{r _n}) = \mD (\delta _{r _1}, \mD (\delta _{r _2} 
\dots \delta _{r _n})) \, .
$
For every multi-index $\alpha$ the operator $\delta^\alpha
=\prod_{1 \leq k \leq d}\delta_k^{\alpha_k}$ and 
its domain $\mD(\delta^\alpha)$ are well defined.
  In analogy to $C^\alpha(\brd)$ we equip $\mD(\delta^\alpha)$ with the norm
$$
\norm{a}_{\mD(\delta^\alpha)} =  \sum_{ \beta \leq 
  \alpha}\norm{\delta^\beta(a)}_\mA\, .$$
Since $\delta _j$ is assumed to be a closed operator on $\mA $, it
follows that $\delta^\alpha$ is a closed operator on $\mD (\delta ^\alpha
)$.

If \mA\ is symmetric and  $\one \in \mD(\delta_k), 1\leq k \leq d$, 
    then $\mD(\delta ^\alpha )$ is \IC\ in \mA.  Furthermore, the \BA\ 
$\mA^{(k)}=\bigcap_{\abs{\alpha} \leq k} \mD(\delta^\alpha)$ and the Fr\'{e}chet algebra $\mA^{(\infty)}=C^\infty(\mA)=\bigcap_{k=0}^\infty \mA^{(k)}$ are \IC\ in \mA ~\cite[3.7]{grkl10}.
   
    \subsection{Automorphism Groups}
    \label{sec:automorphism-groups}

A ($d$-parameter) \emph{automorphism group}
 acting on the \BA\ \mA\   is a set of \BA\ automorphisms $\Psi=\set{\psi_t}_{t \in \brd}$ of
\mA\ that satisfy 
$
  \psi_s \psi_t=\psi_{s+t} \quad \text{for all} \quad s,t \in \br^d
$
and are uniformly bounded, i.e. 
$
M_\Psi= \sup_{t\in \brd}\norm{\psi_t}_{\mA \to \mA} < \infty \, .
$
If \mA\ is a $*$-algebra we assume that $\Psi$ consists of $*$-automorphisms.  

We call $a \in \mA$ \cont\ and write $a \in C(\mA)$, if $\lim_{t \to 0}\psi_t(a)=a$.

For $t \in \brd\setminus\set{0}$ the \emph{generator} $\delta_t$, defined by
$
  \delta_t (a) = \lim_{h \to 0} \frac{\psi_{ht}(a)-a }{h}
$ 
is a closed derivation, and
the domain $\mD(\delta_t, \mA)$ of $\delta_t$ is the set of all $a \in \mA$ for which this limit exists. 
If \mA\ is a $*$-algebra, then $\delta_t$ is symmetric.

The action of $\Psi$ is periodic, if $\psi_t = \psi_{t+e_j}$ for all $t \in \brd$ and all $1 \leq j \leq d$.
If the action of $\Psi$ is periodic, we can define Fourier coefficients of $a \in C(\mA)$ by
\[
\hat a (k) = \int_{\btd} \psi_t(A) e^{-2 \pi i k \cdot t} \, dt
\]

With the group action $\Psi$ it is possible to define the classical smoothness spaces, see, e.g.~\cite{Butzer68}. We need the Besov spaces that are  defined, using the difference operators  $\Delta_t^k=  (\psi_t  -\id)^k$,  by the norm
\[
\norm{a}_{\besov p r\mA} =  \norm{a}_\mA + \Bigl(\int_{\brd} (\abs{t}^{-r}\norm{\Delta^k_t a}_\mA)^p \frac{dt }{\abs t ^d}\Bigr)^{1/p}
\]
for $ 1\leq p \leq \infty$ (standard change for $p=\infty$), $r>0$ and the integer $k > r$ (every choice of $k$ yields an equivalent norm). Algebra properties of $\besov p r \mA$ are discussed in~\cite{klotz10a}. In particular, $\besov p r \mA$ is \IC\ in \mA\ for all $1 \leq p \leq \infty$ and all $r>0$, see~\cite[3.8]{klotz10a}.

In a similar spirit Bessel potential spaces are  introduced and it can be shown that they form \IC\ subalgebras of \mA~\cite{klotz10a}. 
\subsection{Approximation Spaces}
\label{sec:approximation-spaces}
 An \emph{approximation scheme} on the \BA\ \mA\
is a family $(X_n)_{n \in \bn_0}$ of closed subspaces of \mA\ that satisfy
$ X_0=\set{0}$, $ X_n \subseteq X_m$  for $n \leq m$,  and
$  X_n \cdot X_m\subseteq X_{n+m}$, $n,m\in \bn_0$.
If \mA\ is a $*$-algebra, we  assume that
$ \one \in X_1$  and $ X_n=X^*_n$ for all $ n \in \bn_0$.
The \emph{$n$-th approximation error} of  $a \in \mA$ by $X_n$ is
$  E_n(a)=\inf_{x \in X_n} \norm{a-x}_\mA$.
For  $1\leq p < \infty $ and $w$ a weight on $\bno$ the approximation space $\app p w  \mA$  consists of all $a \in \mA$ for which the norm 
\begin{equation}
  \label{eq:appspace}
  \norm{a}_{\mE_w^p}= \bigl(\sum_{k=0}^\infty {E_k(a)^p}w(k)^{p} \bigr)^{1/p}
 \end{equation}
is finite (standard change for $p=\infty$). 
If $w$ is a standard polynomial weight, $w=v_r$ for some $r>0$, then in order to remain consistent with the existing literature we define $\app p r \mA =\app p {v_{r-1/p}} \mA$.

Algebra properties of \as s are discussed in~\cite{Almira06,grkl10}. In particular, in~\cite{grkl10} the following result is proved.
\begin{prop} \label{appspIC} If \mA\ is a symmetric \BA\ with approximation scheme $(X_n)_{ n \in \bn_0}$ then $\mE_r^p(\mA)$ is \IC\ in $\mA$.
\end{prop}

\subsubsection*{Approximation with \BL\ elements}
\label{sec:appr-with-band}
The relation between smoothness and approximation is given by the Weierstrass theorem and Jackson-Bernstein-theorems.

Given a \BA\ with automorphism group we say that $a \in \mA$ is \emph{$\sigma$-bandlimited} for $\sigma >0$, if there is a constant $C$ such that for every multi-index $\alpha$ the Bernstein inequality
  \begin{equation}
    \label{eq:bernstein}
    \norm{\delta^\alpha( a)}_\mA \leq C ( 2 \pi \sigma)^{\abs \alpha} \,
  \end{equation}
   is satisfied. An element is \emph{bandlimited}, if it is $\sigma$-bandlimited for
  some $\sigma
  >0$. 
In this case
$
      X_0=\set{0}$, $ X_n=\set{a \in \mA \colon a \text{ is } n \text{-\BL} }$,  $n \in \bn$,
is an approximation scheme for \mA\ \cite[Lemma 5.8]{grkl10}. 
  \begin{thm}[Weierstrass approximation theorem]\label{proposition:weierstrassBL}
    If $\mA$ is a \BA\  with automorphism group $\Psi$, the set of \BL\ elements is dense in $C(\mA)$.
  \end{thm}
 \begin{thm}[Jackson-Bernstein-Theorem]\label{proposition:jacksonbernstein}
    Let $\mA$ be a \BA\  with automorphism group $\Psi$, and assume that $r >0$ and $1 \leq p \leq \infty$. If $(X_n)_{n \in \bn_0}$ is the approximation
    scheme of bandlimited elements, then $\besov p r \mA = \app p r \mA$. In particular
    \begin{equation}
      \label{eq:cha1}
      a \in \besov  \infty r \mA  \quad \text{if and only if}\quad   E_n(a) \leq C n^{-r} \text{  for all }n >0 \,.
    \end{equation}  
  \end{thm}

\section{Algebras of $C^\infty$ and Ultradifferentiable Elements}
\label{sec:cinfty-ultr-class}

\subsection{$C^\infty$ class}
\label{sec:cinfty-class}
As in the scalar case, elements in a \BA\ with automorphism group that have derivations of all orders can be characterized by approximation properties.
\begin{prop}
  \label{cinfchar}
  If $\mA$ is a \BA\ with automorphism group $\Psi$, and 
 $(X_n)_{n \in \bn_0}$ is the approximation scheme that consists of the \BL\ elements of \mA, then $a \in C^\infty(\mA)$ if and only if for all $r>0$ 
$ 
\lim_{k \to \infty}E_k(a) k^{r} =0 \,.\label{eq:appinforder}
$ 
If the action of $\Psi$ is periodic, this is further equivalent to
$ 
\lim_{\abs k \to \infty}\norm{\hat a (k)}_\mA \abs k ^{r} =0\label{eq:fourdecay}
$ 
for all $r>0$.
\end{prop}
\begin{proof}
  The proof works as for the scalar case. If $a \in C^\infty(\mA)$, then 
$a \in \besov \infty {r+1} \mA$ for any $r>0$ by the properties of Besov spaces~\cite{Butzer67,klotz10a}. Using Proposition~\ref{proposition:jacksonbernstein} we conclude that
$E_k(a) k^{r+1} \leq C$, and $E_k(a) k^{r} \to 0$ for $k \to \infty$. For the other inclusion observe that \eqref{eq:cha1} implies $a \in \besov \infty r \mA$, and further $\delta^\alpha a \in\mA$ for all $\alpha$ with $\abs \alpha = \floor r$, again by the inclusion relations of Besov spaces~\cite{Butzer67,klotz10a}.

If the action of $\Psi$ is periodic, we  use that
for all $b \in X_{\abs k _\infty}$ 
\begin{equation}
  \label{eq:fourcoeffvsapp}
  \hat a (k) = \int_{\btd} (\psi_t(a) -\psi_t(b)) e^ {-2 \pi i k \cdot t}\, dt \,,
\end{equation}
and so 
$
\norm{\hat a (k)}_\mA \leq C \norm{a-b}_\mA 
$.
The infimum of the norm over all $b \in X_{\abs k_\infty}$ yields
 \begin{equation}
  \label{eq:fourierapp}
  \norm{\hat a (k)}_\mA \leq C E_{\abs k_\infty}(a) \, ,
\end{equation}
and so $E_k(a) k^{r} \to 0$ implies $\norm{\hat a(k)}_\mA k^{r} \to 0$. If we assume $\norm{\hat a(k)}_\mA k^{r} \to 0$ for all $r>0$ then 
$\sum_{k \in \bzd}(2 \pi i)^k \hat a(k)$ 
converges in the norm of $\mA$  to $\delta^\alpha(a)$ for all multi-indices $\alpha$, as each $\delta_j$ is closed in $\mD(\delta^\alpha)$, and $a \in C^\infty(\mA)$.
\end{proof}
\subsection{Carleman Classes}
\label{sec:carleman-classes-1}
\begin{defn}[cf. \cite{gorbachuk91,gorbachuk89}] \label{def:carleman}
  Let \mA\ be a \BA\ with commuting derivations $\delta_1,\cdots,\delta_d$, and let  $M=\set{M_k}_{k \in \bn_0}$  be a sequence of positive numbers with $M_0=1$. 
For each $r>0$ we say that $a \in \mA$ is in the Banach space  $C_{r,M}(\mA)$, if the norm
\[ 
\norm{a}_{C_{r,M}(\mA)}= \sup_{\alpha \in \bnd_0}\frac{\norm{\delta^\alpha(a)}_\mA}{r^{\abs\alpha} M_{\abs \alpha}} 
\]
is finite.
The \emph{Carleman} Class $C_M(\mA)$  is the union of the spaces $C_{r,M}(\mA)$,
  \begin{equation*}
    C_M(\mA)= \bigcup_{r>0}C_{r.M}(\mA) \,
  \end{equation*}
with the inductive limit topology. Call $M$ the defining sequence of $C_M(\mA)$.

If $\mA = \bigcap_{j=1}^d \ker \delta_j$ we call $C_M(\mA)$ \emph{trivial}, otherwise $C_M(\mA)$ is \emph{nontrivial}.
\end{defn}
\begin{ex} %
 If $M_k=1$ for all $k$, then $C_{ 2 \pi r,M}(\mA)$ consists of the $r$-\BL\ elements of $\mA$. 
If  $M_k = k!^r$ for $r>0$ then   $\mJ_r(\mA)=C_M(\mA)$ is the \emph{Gevrey-class} of order $r$. In particular,  $\mJ_1(\mA)$ consists of the analytic elements of \mA, i.e., the elements $a \in \mA$  with convergent expansions $\sum_{\alpha \in \bn_0^d}\frac{\delta^\alpha(a)}{\alpha!} t^\alpha$ for some $t >0$. This follows as in the scalar case, see, e.g.~\cite{Timan63}.  Consequently, if $r \leq 1$ then $\mJ_r(\mA)$  consists only of analytic elements.
\end{ex}
\subsubsection*{Equivalence of Defining Sequences}
\label{sec:defining-sequences}

We call
two defining sequences $M$, $N$  \emph{equivalent}, $M \sim N$, if $C_M(\mA)=C_N(A)$. 
If
$
c^k N_k \leq M_k \leq C^k N_k
$
for all indices $k$ and some constants $c,C$ then $M \sim N$.
Fo example, the Gevrey class $\mJ_r$ is also generated by the sequence $N_k= k^{r k}$.

We recall a standard construction.
  Let $M$ be a defining sequence. The \emph{ function associated} to $M$ is
  \begin{equation}\label{eq:assfun}
    T_M(u)=\sup_{k\geq 0}\frac{u^k}{M_k} \quad \text{for } u>0\,.
  \end{equation}
We call $T_N$ and $T_M$ equivalent and write $T_N \sim T_M$, if 
$
T_N(c u) \leq T_M(u) \leq T_N(C u)
$
 for all $u>0$ and some positive constants $c,C$.  A  function associated to the Gevrey class $\mJ_r$ is 
   $
     T_M(u)= \exp(\frac{r}{e} u^{1/r}) 
$.

The \emph{log-convex regularization} $M^c$ of the sequence $M=(M_k)_{k \in \bn_0}$ is the largest logarithmically convex sequence smaller than $M$.

 \begin{prop}[{\cite{komatsu73,Mandelbrojt42,Mandelbrojt52}}]\label{proposition:log-convex-regularization}
   The log-convex regularization of $M$ satisfies
  \begin{equation}
     \label{eq:log-conv-reg}
     M_k^c = \sup_{u>0} \frac{u^k}{T_M(u)} \,.
   \end{equation}
 Moreover, $T_{M^c}=T_M$ and $ M^{cc}=M^c$.
 \end{prop}
We will also need the following simple facts about log-convex sequences.
\begin{lem}[\cite{koosis88,Mandelbrojt52}]\label{convreginc} 
(1) For all $k,l \in \bn_0$ the sequence $M$ satisfies $M^c_k M^c_l \leq M^c_{k+l}$.
(2) The sequence $(M_k^c)^{1/k}$ is increasing.
\end{lem}
 If  $\delta_1, \dotsc, \delta_d$ are generators of an automorphism group we can give a weak type characterization of $C_{r,M}(\mA)$.
\begin{lem}
  \label{proposition-lem-carleman-weak}
  Assume that the automorphism group $\Psi$ acts on \mA. An element $a \in \mA$ is in $C_{r,M}(\mA)$ if and only if $G_{a',a}(t)=\inprod{a',\psi_t(a)},$ is in $C_{r,M}(L^\infty(\brd)))$ for all $a' \in \mA'$, the dual of $\mA$. In this case
$
 \norm{a}_{C_{r,M}(\mA)}\asymp\sup_{\norm{a'}_{\mA'}\leq 1}\norm{G_{a',a}}_{C_{r,M}(L^\infty(\brd))}  \,.
$
\end{lem}
\begin{proof} The required equivalence follows immediately from
  \begin{equation*}\label{eq:xxaa}
\norm{\delta^\alpha a}_\mA
\leq \sup_{\norm{a'}_{\mA'}\leq 1}\norm{ G_{a',\delta^\alpha a}}_{L^\infty(\brd)}
= \sup_{\norm{a'}_{\mA'}\leq 1}\norm{D^\alpha G_{a',a}}_{L^\infty(\brd)}
 \leq M_{\Psi}\norm{ \delta^\alpha a}_\mA  \,
\end{equation*}
by dividing with $r^\alpha M_{\abs \alpha}$ and taking suprema over all $\alpha$. The equality $G_{a',\delta^\alpha a}=D^\alpha G_{a',a}$ is a consequence of elementary properties of $G_{a',a}$~\cite[Lemma 3.20]{grkl10}.
\end{proof}
\begin{prop}[{\cite{gorny39,Mandelbrojt52}}] \label{prop:equivclasses}
  Assume that the automorphism group $\Psi$ acts on \mA, and let  $M$ be a defining sequence for $C_M(\mA)$.
    If $\varliminf M_k^{1/k}=0$, then $C_M(\mA)$ is trivial. 
   If $0 < \varliminf M_k^{1/k} < \infty$, then $C_M(\mA)$ is the class of \BL\ elements.
   If $\lim_{k \to \infty} M_k^{1/k}=\infty$,  and 
    $ 
      \label{eq:defseqequiv}
      (M_k^c)^{1/k} \asymp (N_k^c)^{1/k} \,,
    $ 
      then $C_M(\mA) =C_N(\mA)$.   Moreover, the last condition is equivalent to $T_M \sim T_N$.
 \end{prop}
\begin{proof} 
  As $a \in C_M(\mA)$ if and only if $G_{a',a} \in
  C_M(\brd)$ for all $a' \in \mA'$, the conditions follow from
  \cite[6.5.III]{Mandelbrojt52} by a weak type argument. The statement
  given there is for functions on the real line, but it remains true
  for functions on \brd. In the proof one has to replace the
  Kolmogorov inequality \cite[6.3.III]{Mandelbrojt42} by the
  Cartan-Gorny estimates \cite[(6.4.5)]{Mandelbrojt52}. They can be
  verified for functions on \brd\ as well (see \cite[IV.E.,Problem
  7]{koosis88}).

  The equivalence between condition \eqref{eq:defseqequiv} and $T_M \sim T_N$ follows directly from the definition of equivalent associated functions.
\end{proof}
\begin{cor}
  In particular, we obtain that $C_M(\mA)=C_{M^c}(\mA)$.
\end{cor}

\subsubsection*{Algebra properties of Carleman classes}
\label{sec:algebra-proposition-carl}
In this section we verify that $C_M(\mA)$ is an \IC\ subalgebra of \mA,  if $C_M(\mA)=C_{M^c}(\mA)$. If \mA\ has an automorphism group this follows form Proposition~\ref{prop:equivclasses}.

\begin{prop}
  Each Carleman class $C_M(\mA)$ is an algebra.
\end{prop}
\begin{proof} The proof is as in Komatsu~\cite{komatsu73}.
\end{proof}
  We need the following technical term: A  sequence $(u_k)_{k\in \bn_0}$ of positive numbers is almost increasing, if $u_k \leq C u_l$ for all $k < l$ and a constant $C>0$.
\begin{lem}\label{lem-almost-increasing}
Assume that the defining sequence $M$ satisfies  $M =M^c$. The sequence $(M_k/k!)^{1/k}$ is almost increasing if and only if there is a $C>0$ such that for all $l \in \bn$ and all indices $j_k, k=1,\dotsc,l$ with $j=\sum_{k=1}^l j_k$
  \begin{equation}\label{eq:weakinceq}
\prod_{k=1}^l \frac {M_{j_k}}{j_k!} \leq C^j \frac{ M_j}{j!} \,. 
\end{equation}
\end{lem}
\begin{proof}
 Assuming that $(M_k/k!)^{1/k}$ is almost increasing we obtain
\[
\frac {M_{j_k}}{j_k!} \leq C^k  \Bigl(\frac{ M_{j}}{j!} \Bigr)^{k/j} \,,
\]
and the ``if'' part follows by multiplying these estimates.
For the other implication observe first that Stirling's formula implies that  $(M_k/k!)^{1/k}$ is almost increasing if and only if there is a $C'>0$  such that
\begin{equation} \label{eq:2}
\frac{M_k^{1/k}}{k} \leq C' \frac{M_l^{1/l}}{l} \quad \text{for all  } k < l.
\end{equation}
If $l= rk$ for an integer $r$ then (\ref{eq:weakinceq}) implies
\[
\frac{M_k^{1/k}}{k} \leq C' \frac{M_{rk}^{1/{rk}}}{rk} \,.
\]
If $rk < l < (r+1)k$, we use an interpolation argument.  By Lemma~\ref{convreginc} the sequence $M_k^{1/k}$ is increasing in $k$, so 
\[
\frac{M_l^{1/l}}{l} \geq \frac{M_{kr}^{1/ kr}}{kr} \frac{kr}{l} \geq \frac{kr}{l} \frac{1}{C'} \frac{M_k^{1/k}}{k} \,
\]
by what has been just proved. But this implies
\[
\frac{M_k^{1/k}}{k} \leq C' \frac {l}{kr} \frac{M_l^{1/l}}{l} \leq 2C \frac{M_l^{1/l}}{l} \, . \qed
\]
\end{proof}
\begin{rem}
 (a) For the proof of the direct implication we do  not  need the condition that $M=M^c$. (b) Equation~(\ref{eq:2}) implies that $M_k^{1/k} \to \infty$ if $({M_k}/{k!})^{1/k}$ is almost increasing.
\end{rem}
\begin{thm}[\cite{Malliavin59,Siddiqi90}]\label{cha:gener-carl-classthm-carleman-IC}
  If $C_M(\mA)=C_{M^c}(\mA)$ and  if $({M_k}/{k!})^{1/k}$ is almost increasing, then $C_M(\mA)$ is \IC\ in \mA.
\end{thm}
We adapt the method of~\cite{Siddiqi90} to the noncommutative situation. We need a form of the iterated quotient rule that will be proved in the appendix.
 \begin{lem} \label{invnormderiv}
  Let $E=\set{1, \dotsc, d}$ and $\delta_1,\dotsc, \delta_d$ be derivations
that  satisfy the quotient rule
  \[
  \delta_j(a^{-1}) =-\inv a \delta_j(a) \inv a \quad \text{for all } j
  \in E.
  \]
  For every $k \in \bn$ and every tuple $B=(b_1,\dotsc,b_k) \in E^k$ set 
  $
  \delta_B(a) = \delta_{b_1} \dotsc \delta_{b_k} (a) \,.
  $
  Define the ordered partitions of $B$ into $m$ nonempty subtuples as
  \[
  P(B,m)= \set{(B_1,\cdots,B_m) \colon B=(B_1,\cdots,B_m), B_i \neq
    \emptyset \text{ for all } i} \, .
  \]
   Then
  \begin{equation}
    \label{eq:itquot0}
    \delta_B(\inv a)=
    \sum_{m=1}^{\abs B}(-1)^m \sum_{(B_i)_{1 \leq i \leq m}\in P(B,m)}\Bigl( \prod_{j=1}^m \inv a \delta_{B_i}(a) \Bigr) \inv a \,.
  \end{equation}
    \end{lem}
  
  \begin{proof}[Proof of Theorem~\ref{cha:gener-carl-classthm-carleman-IC}]
    Assume that $\abs \alpha =k$. With the notation of Lemma~\ref{invnormderiv} there is a $k$-tuple $B$ with $\abs B =k$ such that $\delta^\alpha=\delta_B$.
As $a \in C_M(\mA)$,  we know that $\norm{\delta_{B_i}(a)} \leq A r^{\abs{B_i}} M_{\abs{B_i}}$ for some constants $C,r>0$.
The number of (nonempty) partitions of $B$ into sets $(B_i)_{1 \leq i \leq m}\in P(B,m)$ of cardinality $k_i$ is
$
\binom{k}{k_1, \dotsc, k_m} 
$, so we obtain the norm estimate
\begin{equation}
  \label{eq:invnormest}
  \begin{split}
    \norm{\delta^\alpha(\inv a)}_\mA \leq&
    \sum_{m=1}^{k} \norm{\inv a}_\mA^{m+1}\sum_{\substack{k_1+ \dotsb k_m=k \\ k_j \geq 1}} \binom{k}{k_1, \dotsc,k_m}\Bigl( \prod_{j=1}^m C r^{k_j} M_{k_j}
    \Bigr) \\
    =&r^k \sum_{m=1}^{k} \norm{\inv a}_\mA^{m+1}C^m\sum_{\substack{k_1+ \dotsb k_m=k \\ k_j \geq 1}} \binom{k}{k_1, \dotsc,k_m}\Bigl( \prod_{j=1}^m  M_{k_j}  \Bigr)
  \end{split}
\end{equation}
Using (\ref{eq:weakinceq}) we obtain
\[
\begin{split}
\norm{\delta^\alpha (\inv a)}_\mA 
 &\leq r^k C^k M_k\sum_{m=1}^k\norm{\inv a}^{m+1}_\mA A^m \sum_{\substack{k_1+ \dotsb k_m=k \\ k_j \geq 1}} 1 \\
& = r^k C^k M_k\sum_{m=1}^k\norm{\inv a}^{m+1}_\mA A^m \binom{k-1}{m-1}
\leq C_1^k M_k \,,
\end{split}
\]
and this is what we wanted to show.
  \end{proof}
  \begin{cor}
    The Gevrey classes $\mJ_r(\mA)$ are \IC\ in \mA, if $r \geq 1$. 
  \end{cor}

\subsection{Description by Weighted and Approximation Spaces}
\label{sec:conn-with-weights}
In this section we characterize  Carleman classes by unions of weighted spaces and of approximation spaces, if the action of the automorphism group $\Psi$ on the \BA\ \mA\ is periodic and the sequence $M$ satisfies Komatsu's condition (M2').
\begin{defn}\label{defn-weighted}
  Let $\mA$ be a Banach $*$- algebra with periodic automorphism group $\Psi$. For $1 \leq p \leq \infty$ and a weight $v$ on \bzd\ we introduce the weighted spaces spaces
\[
\mC^p_v(\mA) = \set{ a \in A \colon \norm{a}_{\mC^p_v(\mA)}=\Bigl(\sum_{k \in \bzd} \norm{\hat a (k)}_\mA^p v(k)^p \Bigr)^{1/p} < \infty }
\]
with the obvious modification for $p=\infty$, where $\hat a(k)$ are the Fourier coefficients of $a$ (see Section~\ref{sec:automorphism-groups}).
\end{defn}
\begin{rem}\label{proposition:weightedalgebraic}
  If  $\wlp v (\bzd)$ is a \BA\ with respect to convolution, then $\mC^p_v(\mA)$ is an \IC\
  subalgebra of \mA.
The proof is a straightforward adaption of the proof of ~\cite[Theorem 3.2]{GR08}, based on the theorem of Bochner-Philips.
\end{rem}
\begin{lem}\label{proposition:weightinclusion}
If $M$ is a defining sequence for $C_M(\mA)$, $r>0$, and $T_{r,M}(k)=T_M(\frac{2 \pi \abs k_\infty}{r})$, then
$ 
C^1_{T_{r,M}}(\mA) \subseteq C_{r,M}(\mA) \subseteq C^\infty_{T_{r,M}}(\mA) \,.
$ 
\end{lem}
\begin{proof}
  Assume first that $a \in C_{r,M}(\mA)$. 
 Let $j$ be an index such that $\abs{k_j} = \abs{k}_\infty$. Then,  by $l$-fold partial integration
\begin{equation*}
  \hat a (k) = \int_{\btd} \psi_t(a) e^{- 2 \pi i k \cdot t}\, dt = \frac{1}{(2 \pi i k_j)^l }\int_{\btd} \psi_t(\delta_{e_j}^l a) e^{- 2 \pi i k \cdot t}\, dt  \,.
\end{equation*}
 Taking norms we obtain
\begin{equation*}
  \norm{\hat a (k)}_\mA 
  \leq C \frac {r^l M_l}{(2 \pi \abs k _\infty)^l} \,.
\end{equation*}
This relation is valid for all $l \in \bn_0$, and therefore also for the infimum, which yields
$  \norm{\hat a (k)}_\mA \leq C/T_{r,M}(k) $ ,
or $a \in \mC^\infty_{T_{r,M}}(\mA)$.

For the converse inclusion assume that $a \in \mC^1_{T_{r,M}}$, i.e., $\sum_{k \in \bzd} \norm{\hat a (k)}_\mA T_{r,M}(k)<\infty$.
For $\alpha \in \bnd_0$ we  estimate the norm of $\delta^\alpha(a)$ by
\begin{equation*}\label{eq:cmMinctm}
\begin{split}
  \norm{\delta^\alpha(a)}_\mA 
&\leq \sum_{k \in \bzd} \norm{\delta^\alpha(\hat a (k))}_\mA 
\leq \sum_{k \in \bzd} (2 \pi \abs k_\infty)^{\abs \alpha} \norm{\hat a (k)}_\mA\\
&\leq \norm{a}_{\mC^1_{T_{r,M}}(\mA)}\sup_{k \in \bzd} \frac{(2 \pi \abs k_\infty)^{\abs \alpha} }{T_{r,M}(k)} 
\leq \norm{a}_{\mC^1_{T_{r,M}}(\mA)} \sup_{u>0} \frac{u^{\abs \alpha} }{T_M(u/r)} \\
&= \norm{a}_{\mC^1_{T_{r,M}}(\mA)}  r^{\abs \alpha} M_{\abs \alpha }^c  \,,
\end{split}
\end{equation*}
the last equality by \eqref{eq:log-conv-reg}, and so $a \in C_{r,M^c}(\mA)=C_{r,M}(\mA)$.
\end{proof}
\begin{cor}\label{proposition:weightspaceinclusion}
  With the notation of Lemma~\ref{proposition:weightinclusion},
\[
\bigcup_{r>0} \mC^1_{T_{r,M}}(\mA) \inject C_M(\mA) \inject \bigcup_{r>0} \mC^\infty_{T_{r,M}}(\mA) \,,
\]
where all spaces are equipped with their natural inductive limit topologies.
\end{cor}
In order to obtain equality in  Corollary~\ref{proposition:weightspaceinclusion} we impose condition (M2') of Komatsu~\cite{komatsu73}.
\begin{lem}[{\cite{petzsche78},~\cite[Prop. 3.4]{komatsu73}}]\label{proposition:m2weights}
  If  $M$ is a defining sequence, \tfae:
  \begin{enumerate}
  \item  [(M2')] \label{m2d} There exist constants $c>0$, $h>1$ such that for all $k \in \bn$. 
  \item \label{zwei} $T_M(h r) \geq C r T_M(r)$ for all $r>0$.
  \item \label{drei} $\frac{T_M(\lambda r)}{T_M(r)} \geq \exp \bigl(\log(r/c) \log \lambda / \log h \bigr)$ for all $r, \lambda >0$.
  \end{enumerate}
\end{lem}
\begin{ex}
  The defining sequence for the Gevrey-class $\mJ_r$, $r>0$ satisfies (M2').
\end{ex}
\begin{prop}
  \label{proposition:carlemanisweighted.} If \mA\ is a \BA\ with periodic automorphism group and
if the defining sequence satisfies (M2')
\[
 C_M(\mA) =\bigcup_{r>0} \mC^1_{T_{r,M}}(\mA) = \bigcup_{r>0} \mC^\infty_{T_{r,M}}(\mA)= \bigcup_{r>0} \app \infty  {T_{r,M}} \mA \,
\]
with the interpretation that these algebras are topologically isomorphic.
\end{prop}
\begin{proof}{(see, e.g.,\cite{langenbruch89})}
We split the proof into several parts. By known properties of inductive limits \cite{floret68} it is sufficient to prove the following inclusions.
\item [(1)] $ \mC^1_{T_{r,M}}(\mA) \inject \mC^\infty_{T_{r,M}}(\mA)  $ : This follows from Lemma \ref{proposition:weightinclusion}.
\item [(2)] $\mC^\infty_{T_{\lambda r,M}}(\mA) \inject \mC^1_{T_{r,M}}(\mA)$ for some $\lambda >0$ :
 Using Lemma~\ref{proposition:m2weights}, (\ref{drei}), we obtain the estimate 
  \begin{align*}
  \sum_{k \in \bzd} \norm{\hat a (k)}_\mA T_{r,M}(\abs k_\infty) 
  &\leq \sum_{k \in \bzd} \norm{\hat a (k)}_\mA T_{r,M}(\lambda\abs{ k}_\infty) \exp \bigl(-\log(\frac{\abs {k}_\infty}{c}) \frac{\log \lambda}{ \log h} \bigr)\\
  &\leq \sup_{k \in \bzd} \norm{\hat a (k)}_\mA T_{r,M}(\lambda\abs{ k}_\infty) \sum_{k \in \bzd} \Bigl(\frac{\abs {k}_\infty}{c}\Bigr)^{- {\log \lambda}/{ \log h}} \,.
\end{align*}
If we choose $\lambda$ such that $\log \lambda/{ \log h} > d$, the sum on the right hand side of the inequality is convergent.
\item [(3)] $\app \infty  {T_{r,M}} \mA  \inject  \mC^\infty_{T_{ r,M}}(\mA)$ follows from  \eqref{eq:fourierapp}.
\item [(4)]  $  \mC^\infty_{T_{r,M}}(\mA) \inject \app \infty  {T_{ \kappa r,M}} \mA$ for some $\kappa >0$ will be verified \wwlog\ for $r=2 \pi$. 
The approximation error of $a \in \mC^\infty_{T_{2 \pi,M}}(\mA)$ can be estimated by
$
E_l(a) \leq \sum_{\abs k_\infty \geq l}\norm{\hat a (k)}_\mA \leq \norm{a}_{\mC^\infty_{T_{2 \pi,M}}} \sum_{\abs k_\infty \geq l} T_{2 \pi,M}^{-1}(\abs k) \,.
$
As $T_{2 \pi,M}(u)=T_M(u)$ is increasing, we can replace the sum by an integral. We assume that  $l$ is so large that $\frac{\log(l/c) }{ \log h} > 2d$, and obtain 
\begin{align*}
  \sum_{\abs k_\infty \geq l} T_M^{-1}(\abs k) 
  &\leq \int_{\abs k_\infty \geq l} T_M^{-1}(\abs k) \, dk
  \leq C' \int_l^\infty \frac{1}{T_M(u)}u^{d-1} \, du \\
  &= C' l^d \int_1^\infty \frac{1}{T_M(l v )}v^{d-1} \, dv 
   \leq C' \frac{l^d}{T_M(l)} \int_1^\infty v^{d-1} e^{-\frac{\log(l/c) \log v} { \log h}} dv \\
  &= C' \frac{l^d}{T_M(l)} \int_1^\infty v^{ d-1-\frac{\log(l/c)}{ \log h}} dv
  =C' \frac{l^d}{T_M(l)} \frac{1}{\frac{\log(l/c) }{ \log h}-d}  \leq C' \frac{l^d}{T_M(l)} \inv{d}\,,
\end{align*} 
where we have used(\ref{drei}) of Lemma~\ref{proposition:m2weights} in the second line.
Applying (\ref{zwei}) of Lemma~\ref{proposition:m2weights} $d$ times we obtain $T_M(l) \geq C l^d T(l/h^d)$ with a constant $C$ independent of $l$. Substituting this in the current estimate we obtain
$
E_l(a) \leq C\norm{a}_{\mC^\infty_{T_{2 \pi,M}}} T_M^{-1}(l/h^d) \,,
$
and the constant $C$ is independent of $l$. So $\norm{a}_{\mE^\infty_{T_{2 \pi h^d,M}}} \leq C\norm{a}_{\mC^\infty_{T_{2 \pi,M}}}$, and that is what we wanted to show.
\end{proof}
For a more general discussion of approximation results see~\cite{petzsche84}.

\section{Dales-Davie Algebras}
\label{sec:dales-davie-classes}
In this section we assume that  $\Psi$ is a \emph{one} parameter automorphism group acting on the \BA\ \mA. \\

We define {\BA s} that are determined by growth conditions on the sequence $(\norm{\delta^k(a)}_\mA)_{k \in \bn_0}$ by adapting a similar construction introduced in~\cite{dales73} for scalar functions in the complex plane .
\begin{defn}
Let $M =(M_k)_{k\geq0}$  be an \emph{algebra sequence}, that is, a sequence of positive numbers with $M_0=1$ and 
$ 
\frac{M_{k+l}}{(k+l)!} \geq  \frac{M_k}{k!} \frac{M_l}{l!}$ for all  $k,l \in \bn_0$.
The \emph{Dales-Davie algebra} $\dadas 1 M \mA$ consists of the elements $a \in \mA$ with finite norm
\[
\norm{a}_{\dadas 1 M \mA} 
=\sum_{k=0}^\infty \inv{M_k}{\norm{\delta^k(a)}_\mA} \,.
\]
\end{defn}
The space $\dadas 1 M \mA$ is indeed a \BA. This will be proved in Proposition~\ref{proposition-dada}.

\begin{ex}
  Recall that the norm of a derivation on $\mC^1_{v_0}(\mA)$ (see
  Definition \ref{defn-weighted}) is $ \norm{\delta^k(a)}_{
    \mC^1_{v_0}(\mA)}= \sum_{l \in \bz} \norm{\hat a (l)}_\mA (2 \pi
  \abs l)^k \,.  $ For the norm on $\dada {\mC^1_{v_0}(\mA)}$ we
  obtain
  \[
  \norm{a}_{\dada \mA}= \sum_{k=0}^\infty \inv M_k\sum_{l \in \bz}
  \norm{\hat a(l)}_\mA (2 \pi\abs l) ^k = \sum_{l \in \bz} \norm{\hat
    a (l)}_\mA \sum_{k=0}^\infty \frac{(2 \pi\abs l) ^k}{M_k}
  \,.
  \]
  Let us define the weight $v_M$ associated to $M$ by
  \begin{equation}\label{eq:assweight}
    v_M(l)=\sum_{k=0}^\infty \frac{ (2 \pi\abs l) ^k}{M_k} \,.
  \end{equation}
  Thus we obtain $ \dada {\mC^1_{v_0}(\mA)} = \mC^1_{v_M}(\mA) \,.  $
For this example we have established a relation between the growth of derivatives and weights. 
\end{ex}
We recall some notions from complex analysis (see, e.g. \cite{Levin96}). For an entire function $f$ let $M_f(r)= \sup_{\abs x \leq r}\abs{f(x)}$.
The \emph{order} of  $f$ is 
$
\rho_f= \varlimsup_{r \to \infty} {\log \log M_f(r)}/\log r$. 
If $f$ has finite order $\rho_f$, the \emph{type} of $f$ is
$
\sigma_f =\varlimsup_{r \to \infty} {\log M_f(r)}{r^{-\rho_f}} \,.
$
If $\sigma_f=0$, we say that $f$ has minimal type.

In the following lemma  some basic properties of $v_M$ are collected.
\begin{lem}\label{prop:assweights} \hspace{1cm}
  \begin{enumerate}
  \item \label{asw1} If $M$ is an algebra sequence, then $v_M(\abs k)$ is \smult.
  \item \label{asw2} The  weight $v_M$  can be extended from the positive semiaxis to an entire function if and only if $\lim_{k \to \infty}M_k^{1/k}=\infty$.
  \item \label{asw3} The  weight $v_M$  satisfies the GRS condition if and only if 
$
\lim_{k \to \infty}({M_k}/{k!})^{1/k}= \infty 
$.  
Furthermore, $v_M$ is GRS if and only if the analytic continuation of $v_M$ is an entire function of order $\rho_{v_M} \leq 1$, and, if $\rho_{v_M}=1$, then $v_M$ is of minimal type. 
  \end{enumerate}
\end{lem}
\begin{proof}
  (\ref{asw1}) Let $r,s \geq 0$. Then 
  \begin{align*}
    v_M(r+s)
    =\sum_{k=0}^\infty \frac{(2 \pi)^k}{M_k} \sum_{l=0}^k\frac{k!}{l!(k-l)!} r^l s^{k-l}
     \leq \sum_{k=0}^\infty  \sum_{l=0}^k\frac{ (2 \pi r)^l}{M_l} \frac{(2 \pi s)^{k-l}}{M_{k-l}} 
    \leq v_M(r) v_M(s) \,.
  \end{align*}
As $v_M$ is increasing on $\br_0^+$, $v_M(\abs{r+s}) \leq  v_M(\abs r + \abs s)$, and this proves (\ref{asw1}) for all values of $r,s  \geq 0$.

(\ref{asw2}) is a consequence of the Cauchy-Hadamard formula for the convergence radius  of the power series $v_M$.

(\ref{asw3}) We use the following formulas for order and type of the entire function $f(x)=\sum_{k=0}^\infty a_kx^k$ \cite[Theorem 1.2]{Levin96}.
\begin{align}
  \rho_f &=\varlimsup_{k \to \infty} \frac{k \log k}{\log (1/{\abs{a_k})}} ,\label{eq:order}\\
  \sigma_f &= \frac{1}{\rho_f \,e} \varlimsup_{k \to \infty} k \abs {a_k}^{\rho_f/k}  \label{eq:type}\,.
\end{align}
If $v_M$ satisfies the GRS condition then for all $\epsilon >0$ there is a $r(\epsilon)$ such that
$ 
  \label{eq:grsexplicit}
   1 \leq v_M(r) \leq (1+\epsilon)^r=\exp(r\log(1+\epsilon))
$ 
for all $r > r(\epsilon)$. This implies that $\rho_{v_M} \leq 1$, and if $\rho_{v_M} = 1$, then $v_M$  is of minimal type. 
 To verify the last assertion assume first that $\rho_{v_m}<1$ and choose $\epsilon>0$ so small that $\rho_{v_m}+\epsilon <1$. Set $N_k= M_k/(2 \pi)^k$.
Then \eqref{eq:order} implies that 
$
k \log k \leq (\rho_{v_M}+\epsilon) \log{N_k}
$
for $k > k(\epsilon)$, and so
$
k^k < N_k^{\rho_{v_M}+\epsilon} 
$. It follows that
\[
N_k > (k^k)^{\frac{1}{\rho_{v_m}+\epsilon}}= k^{k(1+\delta)}
\]
for some $\delta>0$, and therefore
\begin{equation}\label{eq:mktoinf}
  k^{-1} {N_k^{1/k}} > k^\delta \to \infty
\end{equation}
for $k \to \infty$.
If $\rho_{v_M}=1$, then $v_M$ is of minimal type, and \eqref{eq:type} implies that
\begin{equation}\label{eq:mintype}
  0 =\lim_{k \to \infty} {k}{N_k^{-1/k}} \,,
\end{equation}
and that is what we wanted to show.

If we assume that $\lim_{k \to \infty}(M_k/k!)^{1/k}=\infty$, then the relations~\eqref{eq:mktoinf} and \eqref{eq:mintype} together with~\eqref{eq:order} and \eqref{eq:type} imply that $v_m$ is of order $\leq 1$. If $v_M$ is of order one, the same relations imply that it is of minimal type. This means that for all $\epsilon>0$ there is some $r(\epsilon)$ such that $v_M(r) \leq (1+\epsilon)^r$ for all $r > r(\epsilon)$, so $v_M$ is a GRS weight.
\end{proof}
\begin{prop}\label{proposition-dada}
  If \mA\ is a \BA\ with a  one-parameter group of automorphisms acting on \mA, and $M$ is an algebra sequence, then $\dada \mA$ is a \BA.
\end{prop}
\begin{proof}
  The algebra property follows from using Lemma~\ref{prop:assweights}(\ref{asw1}). To prove completeness let $a_n$ be a Cauchy sequence in $\dada \mA$. This implies that $\delta^k a_n$ is a Cauchy sequence in \mA\ for all indices $k$. As $\delta$ is a closed operator and \mA\ is complete it follows that there is an  $a \in \mA$ such that for all $k\geq0$ the sequence $\delta^ka_n$ converges to $\delta^ka$ in \mA. By standard arguments this implies that $a_n \to a$ in $\dada \mA$.
\end{proof}
\begin{prop}
  If $\mA$ is a \BA\ with  a one-parameter group of automorphisms $\Psi$ acting on \mA, and $M$ is an algebra sequence, then all elements of $\dada \mA$ are \cont, 
$
C(\dada \mA) = \dada \mA
$.
\end{prop}
\begin{proof} For $a \in \dada \mA$,
\[ 
\norm{\psi_t(a)-a}_{\dada \mA} \leq \sum_{k=0}^M \frac{\norm{\psi_t(\delta^k(a))-\delta^k(a)}_{ \mA}}{M_k} + (M_\Psi+1)\sum_{k=M+1}^\infty\frac{\norm{\delta^k(a)}_{ \mA}}{M_k} \,.
\]
For $\epsilon>0$ given we can choose $M$ such that the second sum in the expansion above is smaller than $\epsilon$. As $\delta^k(a) \in C(\mA)$ ~\cite[Proposition.3.15]{grkl10} for all $k \in \bn_0$, the first sum can be made small by choosing $t$ small enough.
\end{proof}
\begin{prop}\label{cor-bl-dense}
The \BL\ elements of \mA\ are dense in $\dada  \mA$.
\end{prop}
\begin{proof}
  We have to verify that the \BL\ elements of \mA\ coincide with the \BL\ elements of $\dada \mA$. Indeed, if $a$ is \BL\ in \mA\ with bandwidth $\sigma$, this means that
$\norm{\delta^ka}_\mA \leq C \sigma^k$ for all $k \in \bn_0$. This implies that
\[
\norm{\delta^k a}_{\dada \mA} = \sum_{l=0}^\infty \frac{\norm{\delta^{l+k}a}_\mA}{M_l} \leq C \sigma^k  \sum_{l=0}^\infty \frac{\norm{\delta^{l}a}_\mA}{M_l}=  C \sigma^k \norm{a}_{\dada \mA} \,,
\]
so $a$ is \BL\ in $\dada \mA$.
Assume now that $a$ is \BL\ in $\dada \mA$. By definition we obtain
$
\norm{\delta^k a}_\mA \leq \norm{\delta^k a}_{\dada \mA} \leq C \sigma^k \,,
$
and so $a$ is \BL\ in \mA. 
An application of  Weierstrass' theorem (Theorem~\ref{proposition:weierstrassBL}) yields the assertion of the proposition.
\end{proof}
\begin{prop}\label{proposition:ddbesov}
$
\dada {\besov 1 r \mA} =\besov 1 r {\dada \mA} \quad \text{for  } r>0.
$  
\end{prop}
\begin{proof}
  Let $a \in \dada {\besov 1 r \mA}$, and assume that $l > \floor{r}$. Then
  \begin{align*}
    \norm{a}_{\dada {\besov 1 r \mA}}
    & = \sum_{k=0}^\infty {M_k^{-1}} \norm{\delta^ka}_ {\besov 1 r \mA}
     = \sum_{k=0}^\infty {M_k^{-1}} \int_0^\infty \frac{\norm{\Delta_t^l\delta^k a}_\mA} {\abs t ^r} \muleb {t} \\
    & = \int_0^\infty \frac{\norm{\Delta_t^l a}_{\dada \mA}} {\abs t ^r} \muleb {t}
    = \norm{a}_{\besov 1 r {\dada \mA}} \,, 
  \end{align*} 
 so $a \in \besov 1 r {\dada \mA} $. The same calculation shows also the converse inclusion.
\end{proof}
As $\app 1 r {\dada \mA}$=$\besov 1 r {\dada \mA}$ by~\eqref{eq:cha1}, Proposition~\ref{proposition:ddbesov} identifies the approximation spaces $\app 1 r {\dada \mA}$ with the Dales-Davie algebras over $\besov 1 r \mA$. \\
\\
As $\mC^1_v(\mA)$ is \IC\ in \mA, if $v$ is a GRS weight it would be natural to conjecture that $\dada \mA$ is \IC\ in \mA\ if and only if  $v_M$ is a GRS weight. However we can only prove the following.
\begin{thm}
  \label{thm-dada-is-IC}
Let $\mA$ be a \BA, and $M$ an algebra sequence. Set $P_k =M_k/k!$.
If 
\begin{equation}\label{eq:icsuff}
A_m= \sup \bigset{\bigl( {P_k}^{-1}{\prod_{j=1}^m P_{l_j}}\bigr)^{1/m} \colon l_j \geq 1 \text{ for } 1 \leq j \leq m, \sum_{j=1}^m l_j =k}
\end{equation}
satisfies $\lim_{m \to \infty}A_m=0$, then $\dada \mA$ is \IC\ in \mA.
\end{thm}
\begin{rem}
  Before proving the theorem we  point out how condition (\ref{eq:icsuff}) is related to the properties of $v_M$.
If (\ref{eq:icsuff}) is valid, it follows in particular that $\lim_{k \to \infty}(M_k/k!)^{1/k}=\infty$ (Choose $l_j=1$ for all $j$ on the RHS of (\ref{eq:icsuff}) ). By Proposition~\ref{prop:assweights} (\ref{asw3}) this means that $v_M$ satisfies the GRS condition.
Let us assume now that  $v_M$ satisfies the GRS condition, or equivalently $\lim_{k \to \infty}P_k^{1/k}=\infty$.  If the sequence $P_k$ is log-convex, then it satisfies (\ref{eq:icsuff}) by ~\cite[Cor. 3.6]{honary07}. 
Choose now a sequence $N_k$ such that $N_k/k!$ is the log-convex minorant of $P_k$ (clearly, $N_k \leq M_k$), then  $N_k/k!$ satisfies (\ref{eq:icsuff}) and is GRS, and $v_{N} \geq v_M$. This means that condition (\ref{eq:icsuff}) does not put stronger growth restrictions on $v_M$ than GRS, but rather imposes some sort of regularity condition.

In this respect it would be interesting to solve the equivalence problem for Dales-Davie algebras: What are the conditions for two algebra sequences $M$ and $N$ such that $\dadas 1 M \mA=\dadas  1 N \mA$?
\end{rem}
\begin{proof}
We use the iterated quotient rule~(\ref{eq:itquot0}) to estimate the norm of $\inv a$.
  \begin{align*}
   \norm{\inv a}_{\dadas 1 M \mA} &= \sum_{k=0}^\infty \frac{\norm{\delta^k(\inv a )}_\mA}{M_k}\\
    &\leq  \norm{\inv a}_\mA +   \sum_{k=0}^\infty \frac{k!}{M_k} \sum_{m=1}^k\norm{\inv a}_\mA^{m+1} \sum_{\substack{l_1+\dotsc l_m=k\\ l_j\geq1}} \prod_{j=1}^m\frac{\norm{\delta^{l_j}(a)}_\mA}{l_j!}\\
    &=  \norm{\inv a}_\mA +   \sum_{k=0}^\infty \frac{k!}{M_k} \sum_{m=1}^k\norm{\inv a}_\mA^{m+1} \sum_{\substack{l_1+\dotsc l_m=k\\ l_j\geq1}} \prod_{j=1}^m P_{l_j}\frac{\norm{\delta^{l_j}(a)}_\mA}{M_{l_j}!}\\
    &=  \norm{\inv a}_\mA +  \sum_{m=1}^k \norm{\inv a}_\mA^{m+1} \sum_{k=m}^\infty  \sum_{\substack{l_1+\dotsc l_m=k\\ l_j\geq1}}  \frac{1}{P_k}\prod_{j=1}^m P_{l_j}\frac{\norm{\delta^{l_j}(a)}_\mA}{M_{l_j}!}\\
    & \leq  \norm{\inv a}_\mA+ \sum_{m=1}^\infty\norm{\inv a}_\mA^{m+1} A_m^m (\norm{a}_{\dadas 1 M \mA} - \norm{a})^m
  \end{align*}
By hypothesis, for any $\epsilon>0$ there exists an index $m_\epsilon$, such that  $A_m < \epsilon$ for all $m > m_\epsilon$.
So if $\epsilon$ is small enough the series converges, and $\norm{\inv  a}_{\dada \mA} < \infty$.
\end{proof}

The condition~\eqref{eq:icsuff} is not easy to verify. Some  sufficient conditions on  $P_k=M_k/k!$ are given in~\cite{honary07}.
\begin{ex}
  The Gevrey sequence  $M_k=k!^r$, $r>1$ is an algebra sequence, and $P_k=k!^{r-1}$ is log-convex. By~\cite[Cor. 3.6]{honary07} this implies (\ref{eq:icsuff}), so  $\dada \mA$ is \IC\ in \mA.
\end{ex}

If the algebra $\mA$ is \emph{commutative}, we can do better by adapting a proof of Hulanicki~\cite{Hulanicki72}.
\begin{prop}\label{proposition:commddic}
  If $\mA$ is a commutative, symmetric \BA , $\Psi$ a periodic one-parameter group of automorphisms acting on \mA,  and $M$ a weight sequence that satisfies $\lim_{k \to \infty}(M_k/k!)^{1/k}=\infty$ (equivalently, $v_M$ is a GRS weight), then $\dada \mA$ is \IC\ in \mA.
\end{prop}
\begin{proof}
  Assume an  $\epsilon >0$, and decompose $a \in \mA$ into
$
a= a_\sigma +r 
$, 
where $\norm{r}_{\dada \mA} < \epsilon$, and $a_\sigma$ is $\sigma$-\BL\ for a $\sigma>0$ that clearly depends on $\epsilon$.
Bernstein's inequality for \BL\ elements (Equation~\eqref{eq:bernstein}) implies that
\begin{equation}
  \label{eq:ddbl}
      \norm{a_\sigma}_{\dada \mA} 
    = \sum_{k=0}^\infty\frac{\norm{\delta^k(a_\sigma)}_\mA}{M_k}
    \leq \sum_{k=0}^\infty\frac{(2 \pi \sigma)^k}{M_k}\norm{a_\sigma}_\mA
= v_M( \sigma) \norm{a_\sigma}_\mA \,.
  \end{equation}
This implies that
\begin{align*}
  \norm{a^p}_{\dada \mA}
  &\leq \sum_{l=0}^p \binom{p}{l} \norm{a_\sigma^l}_{\dada \mA} \epsilon^{p-l} 
  \leq C \sum_{l=0}^p \binom{p}{l} v_M( l \sigma)\norm{a_\sigma}^l_{ \mA} \epsilon^{p-l} \\
  &\leq C v_M( p \sigma)\sum_{l=0}^p \binom{p}{l} \norm{a_\sigma}^l_{ \mA} \epsilon^{p-l} 
   = C v_M( p \sigma) (\norm{a_\sigma}_\mA + \epsilon)^p \\
 & \leq  C v_M( p \sigma) (\norm{a}_\mA + 2\epsilon)^p  \, ,
\end{align*}
 where we have used that $a_\sigma^l$ is $l\sigma$-\BL. So
$
\rho_{\dada \mA}(a)=\lim_{p \to \infty}\norm{a^p}_{\dada \mA}^{1/p} \leq \norm{a}_\mA+2\epsilon 
$, and consequently $\rho_{\dada \mA}(a)= \rho_{ \mA}(a)$.
The Lemma of Hulanicki~\cite[Prop. 2.5]{Hulanicki72} then shows that $\dada \mA$ is \IC\ in \mA.
\end{proof}
\section{Applications to Matrix Algebras with \odd}
\label{sec:matrix-algebras-with}
\subsection{Preliminaries}
In this section we apply the theory developed so far to \IC\ subalgebras of infinite matrices with \odd. This is continuation of~\cite{grkl10,klotz10a}. 

A \emph{\MA} \mA\ (over \bzd) is a \BA\ of matrices that is
continuously embedded in \bopzd.  We drop the reference to the index set
\bzd\ whenever possible.

Our examples are  \MA s  with \odd. One way to describe \odd\ is by weights. 
If \mA\ is a \MA\ and $w$ a weight on \bzd,
  the weighted  space $\mA_w$ consists of the matrices $A \in \mA$ such
  that the matrix $A_w$ with entries 
$ A_w(k,l)= A(k,l)w(k-l)$
is in
  \mA.  The norm on $\mA_w$ is $\norm{A}_{\mA_w}=\norm{A_w}_\mA$.

      If \mA\ is a \MA\ over \bzd,  we define the symmetric and commuting  derivations
      $\delta_j(A)(k,l)= [X_j,A](k,l) = 2 \pi i (k_j-l_j)A(k,l)$, where $X_j=\diag(2 \pi i k_j)_{k \in \bzd}$ and   $1\leq j \leq d$. 
      
      We call \mA\  \emph{solid}, if $A \in \mA$ and $\abs {B(k,l)}
      \leq \abs {A(k,l)}$ for all indices $k,l$ implies $B \in \mA$
      and $\norm{B}_\mA \leq \norm{A}_\mA$. If \mA\ is solid, then 
      $\mA^{(m)}=
      \mA_{v_m}$ (see Section~\ref{sec:derivations} for the definition of $\mA^{(m)}$). In particular, $\mA_{v_m}$ is an \IC\ subalgebra of
      \mA.
    
With the help of the modulation operator $M_t x (k)= \cexp [k \cdot t] x(k)$, $ k\in \bzd$ we define
 a bounded and periodic group action $\chi_t$ on \bop by
\begin{equation*} \chi_t(A) = M_t A M_{-t},\; \chi_t(A)(k,l)=\cexp [(k-l) \cdot t]A(k,l)  \, .
  \end{equation*}
  We call the \MA\ \mA\  \emph{\hmg} (c.~\cite{Deleeuw75,Deleeuw77}, see also \cite[Chapter 9]{Shapiro71}), if the periodic  automorphism group $\chi = \set{ \chi _t}_{t \in \brd}$ is uniformly bounded on \mA. In this case the derivations $\delta_j$ defined above are the generators of the automorphism group.  

If \mA\ is a \HMA, then an easy computation shows that the $k$th Fourier coefficient of $A \in \mA$ coincides with the $k$th side diagonal, so the notation is consistent.

\begin{ex}
(a) The algebra \bop\ itself is a \HMA.  
(b) If \mA\ is solid, then clearly $\mA$ is
  \hmg.  
  (c) In the literature the \HMA s $\cc w=\mC^p_w(\bop)$ are often considered
  (see Definition~\ref{defn-weighted}). More examples can be found in \cite{grkl10,klotz10a}.
\end{ex}

A second possibility to define \MA s with \odd\ is by approximation.
 If \mA\ is a \MA\ and 
$
\mT_N=\set{A \in \mA \colon A=\sum_{\abs k_\infty < N}\hat A (k)}
$
denotes the  matrices in \mA\ with bandwidth smaller than $N$, then $(\mT_N)_{N\geq0}$ is an approximation scheme for \mA.
In this case the algebras $\app p r \mA$ consist of matrices with some kind of \odd\ that in general cannot expressed by weights, see \cite{grkl10,klotz10a}.

If \mA\ is a \HMA, then 
  a $A \in \mA$ is banded with bandwidth $N$,  if and only if it is $N$-bandlimited with
  respect to the group action $\{\chi _t\}$, see~\cite[5.7]{grkl10}.

\subsection{Smooth and ultradifferentiable \MA s }
\label{sec:sumas}
If we apply Proposition~\ref{cinfchar} to \HMA s we obtain the following result.
\begin{cor}
  If \mA\ is a \HMA\,  and $A \in \mA$, then
$\norm{\hat A(k)}_\mA = \bigo (\abs k ^{-r})$ for all $r>0$ if and only if 
          $E_l(A) = \bigo (l^{-r})$ for all $r>0$. 
If $\inv A \in \mA$ then these  conditions imply that
$
\norm{\widehat{\inv A}(k)}_\mA = \bigo (\abs k ^{-r})$ for all $r>0$.
\end{cor}
Considering Carleman classes for a \HMA\  we want to identify the trivial classes first.
 If $\mA$ is a \HMA, then $C_M(\mA)$ is trivial if and only if \mA\ consists only of diagonal matrices. 

It is of some interest, that for a special class of \HMA s we obtain a converse of Theorem~\ref{cha:gener-carl-classthm-carleman-IC} by
 adapting a  construction in~\cite[Thm 1]{Siddiqi90}.
\begin{prop}
  Assume that \mA\ is a \HMA\ and that the translation operators $T_k$ defined by  $(T_k x)(l) = x(l-k)$ are uniformly bounded:
  \begin{equation}\label{eq:transub}
  \norm{T_k}_\mA \leq C
  \end{equation}
 for all $k\in \bzd$. If the nontrivial algebra $C_M(\mA)$ is \IC\ in \mA, then the defining sequence $M$ is almost increasing.
\end{prop}
\begin{rem}
 (a)  Actually, it follows from the proof that itsuffices to assume that  $\norm{T_{ke_1}}_\mA \leq C$ for all $k\in \bno$. (b) The condition~(\ref{eq:transub}) is equivalent to $\mC^1_{v_0} \inject \mA$, where $\mC^1_{v_0}$ denotes the (unweighted) Baskakov algebra, see Section~\ref{sec:matrix-algebras-with}.
\end{rem}
\begin{proof}
  \Wlog\ we may assume that $M=(M_k)_{k \in \bno}$ is log-convex. By Proposition \ref{prop:equivclasses}, the condition  $\varliminf M_k =0$ implies that $C_M(\mA)$ is trivial. If  $\varliminf M_k <\infty$ the algebra $C_M(\mA)$ consists of the banded matrices, which are not \IC\ in \mA. So we may assume that  $\lim_{k \to \infty}M_k^{1/k}=\infty$.
Using the convexity  polygon with vertices $\bigl((k, \log k)\bigr)_{k \geq 0}$  we can find  positive integers $(u_j)_{j \in \bno}$ with the property that 
$
T_M(u_j)=\inv{M_j} u_j^j
$
(\cite{Mandelbrojt52}, see also \cite{koosis88} for a detailed discussion). The matrix
$
A=\sum_{j = 1}^\infty 2^{-j} T^{-1}_M(u_j) T_{u_j e_1} 
$
satisfies
\[
\norm{\delta_1^m(A)}_\mA \leq C \sum_{j=1}\frac{(2 \pi)^m u_j^m}{ 2 ^j T_M(u_j)} \leq C'(2 \pi)^m M_m 
\]
for a constant $C' >0$.  If we choose $\lambda > 1 +\norm{A}_\mA$ then $(\lambda - A)$ is invertible in \mA, and, by hypotheses, $\inv{(\lambda - A)} \in C_M(\mA)$. This means that $\norm{\delta_1^m \inv{(\lambda - A)} }_\mA < C q^m M_m$  for all $m \in \bno$ and for constants $C, q >0$.
As \mA\ is a \MA\ it follows that   $\norm{\delta_1^m \inv{(\lambda - A)} }_{\bop} < C q^m M_m$. As $A$ is constant along the diagonals standard facts on convolution operators imply that 
 $\norm{\delta_1^m \inv{(\lambda - A)} }_{\bop}= \norm{D^m_1\inv{(\lambda -f)}}_{L^\infty(\btd)}$, where
\[
f(t)=\sum_{j = 1}^\infty 2^{-j} T^{-1}_M(u_j) \cexp [u_j \cdot t_1] \,.
\]
As in ~\cite[Thm 1]{Siddiqi90} we can conclude that $D_1^m f(0)= i^m s_m$, where $s_m > K^{-m} M_m$ for a $K > 1$.
Now let us consider the expansion, given by the iterated quotient rule
\begin{align*}
D^m_1 (\lambda -f)^{-1}(0)&=\sum_{l=1}^m(\lambda-f)^{-l-1}(0) \sum_{\substack{k_1+ \dotsb k_l=m \\ k_j \geq 1}} \binom{k}{k_1, \dotsc,k_l} \prod_{j=1}^l D_1^{k_i}f (0)   \\
&=i^mK^{-m}\sum_{l=1}^m(\lambda-f)^{-l-1}(0) \sum_{\substack{k_1+ \dotsb k_l=m \\ k_j \geq 1}} \binom{k}{k_1, \dotsc,k_l} \prod_{j=1}^l s_{i_k}
\end{align*}
 As $\lambda-f >0$,  we conclude that
 \begin{equation*}
   \sum_{l=1}^m(\lambda-f)^{-l-1}(0) \sum_{\substack{k_1+ \dotsb k_l=m \\ k_j \geq 1}} \binom{k}{k_1, \dotsc,k_l} \prod_{j=1}^l M_{i_k} \leq  C q^m M_m \,, 
 \end{equation*}
and, in particular 
\begin{equation*}
  \binom{k}{k_1, \dotsc,k_l} \prod_{j=1}^l M_{i_k} \leq  C q^m M_m \,,
\end{equation*}
for $k_1+ \dotsb k_l=m$ (actually, we have assumed that $k_j \geq 1$ for all indices $j$, but the reader can easily verify the last relation for $k_j=0$ as well). By Lemma~\ref{lem-almost-increasing} this means that $M$ is almost increasing.
\end{proof}
\begin{cor}
  Assume that \mA\ is a \HMA\ with uniformly bounded translations. Any  nontrivial Carleman class  $C_M(\mA)$ is \IC\ in \mA, if and only if $(M_k/k!)^{1/k}$ is almost increasing.
\end{cor}

The characterization of Carleman classes by weighted spaces and by approximation spaces in Proposition~\ref{proposition:carlemanisweighted.} yields a new  proof of a result of Demko, Smith and Moss~\cite{Demko84} and of Jaffard~\cite{Jaffard90}.
\begin{cor}
  Assume that $A \in \bop$ has  \odd\ of  exponential order: $\abs{{A(k,l)}} \leq C e^{-\gamma \abs{ k-l}^r }$ for constants $C,\gamma >0$, $0 < r \leq 1$, and all $k,l, \in \bzd$. If $\inv A \in \bop$, then there exist $C',\gamma'>0$ such that
  \[
\abs{\inv{A}(k,l)} \leq C' e^{-\gamma' \abs {k-l}^r} \quad \text{for all } k,l \in \bzd.
\]
\end{cor}
\begin{rem}
  If $r < 1$ better results exist~\cite{Baskakov97,GL04a}.  
\end{rem}

Dales-Davie algebras of matrices describe new forms of \odd. It has already been  stated that $\mC^1_{v_M}=\dada {\mC^1_{v_0}}$. Now consider the following simple corollary of Theorem~\ref{thm-dada-is-IC}.
  \begin{cor}
    If $A \in \bopz$ satisfies
    \[
    \sum_{k=0}^\infty (k!)^{-r} \norm{\delta^k A}_{\bop(\bz)} < \infty \,
    \]
for a $r>1$, and $A$ has an inverse $\inv A \in \bopz$, then this inverse satisfies the same estimate.
  \end{cor}
The result remains valid, if we replace $(k!)^{-r}$ by any algebra sequence $M$, such that $M_k/k!$ is log-convex and $\lim_{k \to \infty}(M_k/k!)^{1/k}=\infty$.

\appendix

\section*{An Iterated Quotient rule}

We provide a proof of Lemma~\ref{invnormderiv}.
The proof is by induction over $\abs B$. If $\abs B = 1$ there is nothing to prove.
Assume that the statement is true for $\abs B < k$, and assume $\abs B = k$.
The Leibniz rule for $\delta_B(\inv a a)$ yields
\begin{equation}
  \label{eq:leibnizpart}
  \delta_B(\inv a a)=0=
  \sum_{(B_1,B_2) \in P(B,2)}\delta_{B_1}(\inv a) \delta_{B_2}(a)+ \inv a \delta_B(a) + \delta_B(\inv a) a \,,
\end{equation}
So 
\begin{equation}
  \label{eq:quot1}
  \delta_B(\inv a) = - \inv a \delta_B(a) \inv a
                  - \sum_{(B_1,B_2) \in P(B,2)}\delta_{B_1}(\inv a) \delta_{B_2}(a) \inv a  \,.
\end{equation}

As $\abs{B_1}<k$ we can apply the induction hypothesis.
\begin{align*}
  \delta_B(\inv a) = &- \inv a \delta_B(a) \inv a\\
    &-\sum_{(B_1,B_2) \in P(B,2)} 
\sum_{m=1}^{\abs{ B_1}}(-1)^m \sum_{(D_i)_{1 \leq i \leq m}\in P(B_1,m)}\Bigl( \prod_{j=1}^m \inv a \delta_{D_i}(a) \Bigr) \inv a \, \delta_{B_2}(a) \inv a  \,.
\end{align*}
Interchanging the first two summations we obtain
\begin{align*}
  \delta_B(\inv a) = &- \inv a \delta_B(a) \inv a\\
    &-\sum_{m=1}^{k-1}(-1)^m \sum_{\substack{(B_1,B_2) \in P(B,2)\\ \abs{B_1}\geq m}} 
 \sum_{(D_i)_{1 \leq i \leq m}\in P(B_1,m)}\Bigl( \prod_{j=1}^m \inv a \delta_{D_i}(a) \Bigr) \inv a \, \delta_{B_2}(a) \inv a  \,.
\end{align*}
Now observe that in this expression $(D_1,\cdots,D_m,B_2)$ varies over all partitions of $B$, The condition 
$(D_i)_{1 \leq i \leq m}\in P(B_1,m)$
already implies that $\abs{B_1} \geq m$, so we can set $D_{m+1}=B_2$ and obtain
\begin{align*}
  \delta_B(\inv a) = &- \inv a \delta_B(a) \inv a\\
    &-\sum_{m=1}^{k-1}(-1)^m  
 \sum_{(D_i)_{1 \leq i \leq m+1}\in P(B,m+1)}\Bigl( \prod_{j=1}^{m+1} \inv a \delta_{D_i}(a) \Bigr) \inv a \, .
\end{align*}
We change the summation index.
\begin{align*}
  \delta_B(\inv a) = &- \inv a \delta_B(a) \inv a\\
    &+\sum_{l=2}^{k}(-1)^m  
 \sum_{(D_i)_{1 \leq i \leq l}\in P(B,l)}\Bigl( \prod_{j=1}^{l} \inv a \delta_{D_i}(a) \Bigr) \inv a \, .
\end{align*}
The term $- \inv a \delta_B(a) \inv a$ can be included into the sum for $l=1$. We obtain
\[
\delta_B(\inv a) = \sum_{l=1}^{k}(-1)^m  
 \sum_{(D_i)_{1 \leq i \leq l}\in P(B,l)}\Bigl( \prod_{j=1}^{l} \inv a \delta_{D_i}(a) \Bigr) \inv a \, ,
\]
and this is \eqref{eq:itquot0}. 
\bibliographystyle{abbrv}
\def\cprime{$'$} \def\cprime{$'$} \def\cprime{$'$}

\end{document}